\documentclass[a4paper]{amsart}
\usepackage{amssymb}
\usepackage{verbatim}
\usepackage{amscd}
\def\card{{{\operatorname{card}}}}
\def\tr{{{\operatorname{tr}}}}
\numberwithin{equation}{section}
\theoremstyle{plain}
\newtheorem{theorem}[equation]{Theorem}

\newtheorem{proposition}[equation]{Proposition}
\newtheorem{lemma}[equation]{Lemma}
\newtheorem*{(DQ1)}{(DQ1)}
\theoremstyle{definition}

\theoremstyle{remark}

\begin{document}
\title{Excluding words from Dyck shifts}
\author{Kokoro Inoue}
\author{Wolfgang Krieger}
\begin{abstract} We study subshift that arise by excluding words of length two from Dyck shifts. 
The words that are to be excluded are taken from a finite set that is not literal-uniform.
 \end{abstract}
 
\maketitle

\section{Introduction}
Let $\Sigma$ be a finite alphabet, and let $S$ be the shift 
on the shift space $\Sigma^{\Bbb Z}$,

$$
S((x_{i})_{i \in \Bbb Z}) =  (x_{i+1})_{i \in \Bbb Z}, 
\qquad 
(x_{i})_{i
\in \Bbb Z}  \in \Sigma^{\Bbb Z}.
$$
An 
$S$ -invariant closed subset $X$ of $\Sigma^{\Bbb Z}$ is called a subshift. For an introduction to the theory 
of subshifts see \cite {Ki} or  \cite{LM}.
A word is called admissible for the subshift $X \subset  \Sigma^{\Bbb Z}$ if it appears in a point of $X$.  We  denote the language of admissible words of a the subshift $X \subset  \Sigma^{\Bbb Z}$by $\mathcal L (X)$. 
A basic class of subshifts are the subshifts of finite type. A subshift of finite type is  constructed from a finite set $\mathcal F$ of words in the alphabet $\Sigma$ as the subshift that contains the points in $\Sigma^{\Bbb Z}$, in which no word in $\mathcal F$ appears. More generally, a subshift $X\subset\Sigma^{\Bbb Z}$ and 
a finite set $\mathcal F \subset \mathcal L (X)$ determines a subshift $X(\mathcal F)$ that contains the points in $X$ in which no word in $\mathcal F$ appears. We say that the subshift $X(\mathcal F)$ arises from the subshift $X$ by excluding words. In this paper we study subshifts that arise from Dyck shifts by excluding words of length two..
To recall the construction of the Dyck shifts, let $N> 1$, and let $\alpha^-(n), \alpha^+(n), 0 \leq n < N,$ be the generators of the Dyck inverse monoid \cite {NP} $\mathcal D_N$ with the rules
$$
\alpha^-(n) \alpha^+(n^\prime) =
\begin{cases}
1, &\text{if  $n = n^\prime$}, \\
0, &\text {if $n \neq n^\prime$}.
\end{cases}
$$
The Dyck shifts  are defined as the subshifts
$$
D_N \subset( \{ \alpha^-(n): 0 \leq n < N \} \cup\{ \alpha^+(n):0 \leq n < N \})^\Bbb Z
$$
with the admissible words $(\sigma_i)_{1 \leq i \leq I  } , I \in \Bbb N,$ of $D_N, N > 1,$ given by the condition 
$$
\prod_{1 \leq i \leq I } \sigma_i \neq 0.
$$
We denote by $C_N(n)$  the code that contains the words in $\mathcal L(D_N)$ that have $\alpha^-(n)$ as their  first symbol and  $\alpha^+(n)$ as their last symbol, and that have no proper prefix with $\alpha^+(n)$ as the last symbol, or, equivalently, that have no proper suffix with $\alpha^-(n)$ as the first symbol, $0 \leq n < N$.
The Dyck shift $D_N$ can also be defined as the coded system \cite {BH} of the Dyck code $\cup_{0 \leq n<N}C_N(n)$.
In \cite{HI} a necessary and sufficient condition was given for the existence of an embedding of an irreducible subshift of finite type into a Dyck shift. In \cite {HIK} this result was extended to a wider class of target shifts that have presentations that were constructed by using  a graph inverse semigroup $\mathcal S$. These presentations were called  $\mathcal S$-presentations. The Dyck inverse monoids occupy a central place among the graph inverse semigroups.
With  the semigroup $\mathcal D_N^-$ ($\mathcal D_N^+$) that is generated by $ \{ \alpha^-(n): 0 \leq n < N \} $ $ ( \{ \alpha^+(n): 0 \leq n < N \} )$, 
$\mathcal D_N$-presentations can be described as arising from a finite irreducible directed labelled graph with vertex set $\mathcal V$ and edge set $\Sigma$ and a label map $\lambda$, such that 
$$
\lambda(\sigma) \in \mathcal D_N^- \cup \{1\} \cup\mathcal D_N^+.
$$
Extending the label map to paths 
$
b = (b_i)_{1\leq i \leq I}, I > 1,
$
in the directed graph 
by setting
$
\lambda (b) = \prod_{1\leq i \leq I} \lambda(b_i),
$
the admissible words of the $\mathcal D_N$-presentations are the paths $b$ in the directed graph that satisfy  the condition
$
\lambda (b) \neq 0.
$
It is required that one has for  $U, W \in \mathcal V$, and for $\beta \in \mathcal D_N,$ that in  the directed graph there is  a path $b$ from $U$ to $W$ such that $\lambda(b) =Ê\beta$. 
A periodic point $p =(p_i)_{i \in \Bbb Z}$ of a $\mathcal D_N$-presentation is said to have non-positive (non-negative) multplier, if, with $\Pi(p)$ the period of $p$ there exists an $i \in \Bbb Z$ such that $\lambda( (p_j)_{i\leq j < i+\Pi(p)})\in$ $\mathcal D_N^-\cup \{ \bold 1 \}$$ (\{ \bold 1 \}\cup \mathcal D_N^+)$.
Among the invariants, that determine the existence of an embedding of a given irreducible subshift of finite type into a $\mathcal D_N$-presentation, are periodic point counts and entropies that are associated to the periodic points of  target shift with non-positive or non-negative multipliers. 
Examples of $\mathcal D_N$-presentations can be obtained by excluding finitely many words from Dyck shifts. The examples in Section 4 of \cite {HIK} were constructed in this way. This then leads to the problem of determining zeta functions  that are associated to subshifts that are obtained by  excluding from a Dyck shift $D_N$ the words in a finite set $\mathcal F$ of  $D_N$-admissible words. In studying this problem one is lead to make a distinction according to the nature of the set $\mathcal F$. In \cite {IK} a $D_N$-admissible word was called literal-non-positive (literal-non-negative) if all of its symbols are in
$\{       \alpha^-(n):0 \leq n < N       \}$ ( $\{        \alpha^+(n): 0 \leq n < N       \}  )$, and a set of  $D_N$-admissible words was called literal uniform if all of its words are  literal-non-positive or literal-non-negative. In \cite {IK} we considered subshifts that are obtained by excluding from a Dyck shift $D_N$ the words in a finite literal-uniform set of  $D_N$-admissible words, 
and in this paper we consider subshifts that are obtained by excluding from a Dyck shift $D_N$ the words in a  set  of  $D_N$-admissible words of length two that is not literal-uniform (literal non-uniform). More generally, we study subshifts that arise by excluding words from certain subshifts that belong to a class of subshifts, that are constructed from an finite index set $  \Gamma$ and a relation $\sim$ on $\Gamma$, with the Dyck shifts as special cases. This wider class of subshifts contains the subshifts $X(  \Gamma,  \sim )$  with alphabet
 $$
\{\alpha^-(\gamma): \gamma \in \Gamma\} \cup\{\alpha^+(\gamma): \gamma \in \Gamma\}
$$
 and admissible words  $(\sigma_i)_{1 \leq i \leq I  } $, that are given, with the rules
 $$
\alpha^-(\gamma) \alpha^+(\gamma^\prime) =
\begin{cases}
1, &\text{if  $\gamma \sim \gamma^\prime$}, \\
0, &\text {if $\gamma\nsim \gamma^\prime$},
\end{cases}
$$
by the condition 
$$
\prod_{1 \leq i \leq I } \sigma_i \neq 0,
$$
(see \cite [Section 4]{HK}). The notions of multiplier, and of a literal-uniform set and literal-non-uniform set of words, for the subshifts $X(  \Gamma,  \sim )$ are analogous to the ones for the Dyck shifts.
   
In this paper, more specifically, we  let $N >1$, choose
$M_n \in \Bbb N, 0 \leq n < N$, set
$$
\Gamma = \bigcup _{0 \leq n <N}\{ (n, m): 1 \leq m \leq M_n  \},
$$
and use the relation $\sim$, where
$(n,m)\sim(n^\prime,m^\prime)$ means that $n = n^\prime$, denoting the resulting subshift $X(  \Gamma,  \sim )$ by $X_{N}((M_n)_{0\leq  n <N})$
 (The subshift $X_{N}((1)_{0\leq  n <N})$ is $D_N$). Moreover,
we choose  sets
$$
\mathcal A^-_{n, n^\prime}\subset [1, M_{n^\prime}],  \ \
 \mathcal A_{n, n^\prime}\subset [1, M_{n^\prime}], \ \
\mathcal A^+_{n, n^\prime}\subset [1, M_{n^\prime}], \  \  \qquad  0 \leq n,n^\prime < N,
$$
and set
\begin{align*}
\mathcal F((M_n)_{0 \leq n < N},(&\mathcal A^-_{n, n^\prime},\mathcal A_{n, n^\prime},\mathcal A^+_{n, n^\prime})_{Ú0 \leq n,n^\prime < N} )=\\
\bigcup_{ 0 \leq n, n^\prime < N } \bigcup_{ 1 \leq m \leq M_n }
(&\{\alpha^-(n, m) \alpha^-(n^\prime, m^\prime): m^\prime \notin \mathcal A^-_{n, n^\prime}\}\cup\\
&\{ \alpha^+(n, m) \alpha^-(n^\prime, m^\prime ):
m^\prime \notin \mathcal A _{n, n^\prime} \}\cup\\
&\{\alpha^+(n, m)  \alpha^+(n^\prime, m^\prime):m^\prime \notin \mathcal A^+_{n, n^\prime}\}).
\end{align*}
 We also set
 \begin{multline*}
 X((M_n)_{0 \leq n < N},(\mathcal A^-_{n, n^\prime},\mathcal A_{n, n^\prime},\mathcal A^+_{n, n^\prime})_{Ú0 \leq n,n^\prime < N}) =\\
 X_{N}((M_n)_{0\leq  n <N})(\mathcal F((M_n)_{0 \leq n < N},(\mathcal A^-_{n, n^\prime},\mathcal A_{n, n^\prime},\mathcal A^+_{n, n^\prime})_{Ú0 \leq n,n^\prime < N} )).
\end{multline*}
(There are procedures to decide if these subshifts
are empty, or of finite type, or topologically intransitive.)
By the use of circular Markov codes and by applying a formula of Keller \cite {Ke}, we obtain in Section 2 an expression for the zeta function of the subshifts
$
 X((M_n)_{0 \leq n < N},(\mathcal A^-_{n, n^\prime},\mathcal A_{n, n^\prime},\mathcal A^+_{n, n^\prime})_{Ú0 \leq n,n^\prime < N}).
$

In section 3 we consider the case $N=2$. The equations for the generating functions of the two essential  circular Markov codes of the subshift are in this case of degree at most four, and in the case, that the two generating functions are equal, these equations are of degree at most three. The recursive structure of the code words allows to
show that the two generating functions are the same, provided the number of words of length two, four and six, that the codes contain, are the same. 

In section 4 we specialize to the case of the Dyck shift $D_2$. We  determine the literal non-uniform sets  $\mathcal F$ of words of length two, such that the subshifts, that arise from $D_2$  by removing the word in $\mathcal F$, have a 2-block system that yields an  $D_2$-presentation, such that the two essential  circular Markov codes have the same generating function.

\section{Subsystems of $X((M_n)_{0 \leq n < N})$}
 
 We denote the generating function of  a set  $\mathcal C$ of words
in the symbols of a finite alphabet by $g( \mathcal C )$.

We recall from  \cite{Ke} the notion of a circular Markov code to the extent that is needed here.
We let a Markov code  be given by a nonempty set  $\mathcal C$ of words
in the symbols of a finite alphabet $\Sigma$
together with a finite set ${\mathcal V}$, a 0\thinspace-1 transition matrix $B=(B( U, W)_{U, W \in \mathcal V}$ and mappings
$r:\mathcal C \to\mathcal  V, s:\mathcal C \to\mathcal V$.
To
$(\mathcal C,r,s)$ there is associated the shift invariant set 
$X_\mathcal C \subset \Sigma^{\Bbb Z}$
of points $x \in \Sigma^{\Bbb Z}$ such that 
there are indices 
 $I_k, k \in {\Bbb Z},$ such that 
$$
I_0 \le 0 < I_1,\quad  I_k < I_{k+1}, \quad k \in {\Bbb Z}, 
$$
and such that
\begin{align*}
x_{[I_k,I_{k+1})} \in \mathcal C, \qquad k \in {\Bbb Z},  \tag {2.1}
\end{align*}
and
\begin{align*}
B(r(x_{[I_k,I_{k+1})}) ,s(x_{[I_{k+1},I_{k+2})}))= 1, \qquad k \in {\Bbb Z}.  \tag {2.2}
\end{align*}
$(\mathcal C, r,s,B)$ is said to be a circular Markov code
if for every periodic point $x$ in $X_\mathcal C$ 
the indices $I_k, k \in {\Bbb Z},$ 
such that  (2.1)
and  (2.2) hold,
are uniquely determined by $x$.
Given a circular Markov code
$(\mathcal C, s, r,A)$ denote by
$
\mathcal C(U,W)
$
the set of words $c \in \mathcal C$ such that
$s(c) = U$,  $r(c) = W,U,W \in {\mathcal V}.$
Introduce the matrix
$$
H^{(\mathcal C)}(z) =
 (B(U, W)g_{\mathcal C (U,W)}(z))_{U,W \in {\mathcal V}}.
$$
 
 \begin{lemma} 
For a circular Markov code $(\mathcal C, s, r ,B),$  
\begin{align*}
\zeta_{X_\mathcal C}(z) = \det ( I - H^{(\mathcal C)}(z))^{-1}. 
\end{align*}
\end{lemma}
\begin{proof}
This is a variant of a special case of a formula of Keller \cite {Ke}.
\end{proof}
We return to the subshifts
$
X_N((M_n)_{0 \leq n < N},(\mathcal A^-_{n, n^\prime},\mathcal A_{n, n^\prime},\mathcal A^+_{n, n^\prime})_{Ú0 \leq n,n^\prime < N}).
$

 \begin{proposition} 
 Let $N>1$, and $M_n\in \Bbb N,  0 \leq n < N$, and let
  \begin{multline*}
\mathcal A^-_{n, n^\prime},\mathcal A^-_{n, n^\prime}\subset [1, M_{n^\prime}],  \ \
 \mathcal A_{n, n^\prime}, \mathcal A_{n, n^\prime}\subset [1, M_{n^\prime}],  \ \
\mathcal A^+_{n, n^\prime},\mathcal A^+_{n, n^\prime}\subset [1, M_{n^\prime}],  \\
 0 \leq n,n^\prime < N, 
 \end{multline*}
 be such that
 \begin{multline*}
 \card  \  \mathcal A^-_{n, n^\prime} = \card  \ \widetilde {\mathcal A}^-_{n, n^\prime} ,
 \card  \  \mathcal A_{n, n^\prime} = \card  \ \widetilde {\mathcal A}_{n, n^\prime} ,
 \card  \  \mathcal A^+_{n, n^\prime} = \card  \ \widetilde{ \mathcal A}^+_{n, n^\prime} , \\
  0 \leq n,n^\prime < N.
 \end{multline*}
  Then the subshifts
  $$
X_N((M_n)_{0 \leq n < N},(\mathcal A^-_{n, n^\prime},\mathcal A_{n, n^\prime},\mathcal A^+_{n, n^\prime})_{Ú0 \leq n,n^\prime < N})
$$
and
$$
X_N((M_n)_{0 \leq n < N},\widetilde {\mathcal A}^-_{n, n^\prime},\widetilde {\mathcal A}_{n, n^\prime},\widetilde {\mathcal A}^+_{n, n^\prime})_{Ú0 \leq n,n^\prime < N})
$$
are topologically conjugate.
\end{proposition}
 \begin{proof}
 With permutations
 $
 \Psi^-_{n, n^\prime},
 \Psi_{n, n^\prime},
 \Psi^+_{n, n^\prime}, 0 \leq n < N,
 $ 
 of $[1, M_{n^\prime}], \  0 \leq n^\prime < N$, such that
 \begin{multline*}
 \Psi^-(n, n^\prime)(\mathcal A^-_{n, n^\prime} ) = \widetilde {\mathcal A}^-_{n, n^\prime}, \
 \Psi(n, n^\prime)( \mathcal A_{n, n^\prime}) =\widetilde {\mathcal A}_{n, n^\prime}, \
 \Psi^+(n, n^\prime)( \mathcal A^+_{n, n^\prime}) =\widetilde {\mathcal A}^+_{n, n^\prime},\\
 0 \leq n,n^\prime < N,
 \end{multline*}
a topological conjugacy of \
 $
X_N((M_n)_{0 \leq n < N},(\mathcal A^-_{n, n^\prime},\mathcal A_{n, n^\prime},\mathcal A^+_{n, n^\prime})_{Ú0 \leq n,n^\prime < N})
$
\ onto 
$
X_N((M_n)_{0 \leq n < N},\widetilde {\mathcal A}^-_{n, n^\prime},\widetilde {\mathcal A}_{n, n^\prime},\widetilde {\mathcal A}^+_{n, n^\prime})_{Ú0 \leq n,n^\prime < N})
$
is given by the mapping that replaces in a point of \
 $
X_N((M_n)_{0 \leq n < N},(\mathcal A^-_{n, n^\prime},\mathcal A_{n, n^\prime},\mathcal A^+_{n, n^\prime})_{Ú0 \leq n,n^\prime < N})
$ \
a symbol
$\alpha^-(n^\prime,m^\prime)$, if preceded by the symbol
$\alpha^-(n, m)$, in which case $m^\prime\in \mathcal A^-_{n, n^\prime}   $,
by the symbol $\alpha^-(n^\prime, \Psi^-_{n, n^\prime}(m^\prime))$,
a symbol
$\alpha(n^\prime,m^\prime)$, if preceded by the symbol
$\alpha(n, m)$, in which case $m^\prime\in \mathcal A_{n, n^\prime}   $,
by the symbol $\alpha(n^\prime, \Psi^-_{n, n^\prime}(m^\prime))$,
a symbol
$\alpha^+(n^\prime,m^\prime)$, if preceded by the symbol
$\alpha^+(n, m)$, in which case $m^\prime\in \mathcal A^+_{n, n^\prime}   $,
by the symbol $\alpha^+(n^\prime, \Psi^-_{n, n^\prime}(m^\prime))$,
 \end{proof}

Setting
\begin{multline*}
A^-(n,n^\prime)= \card  \  \mathcal A^-_{n, n^\prime}, \ \ A(n,n^\prime)= \card  \  \mathcal A_{n, n^\prime}, \ \ A^+(n,n^\prime)= \card  \  \mathcal A^+_{n, n^\prime},\\ 0\leq n,n^\prime < N,
\end{multline*}
we introduce matrices
$$
A^- = (A^-(n,n^\prime))_{0\leq n,n^\prime < N}, A = (A(n,n^\prime))_{0\leq n,n^\prime < N},  A^+ = (A^+(n,n^\prime))_{0\leq n,n^\prime < N}.
$$
In view of Proposition (2.2) we will will write 
$X_N((M_n)_{0 \leq n < N},A^-,A,A^+) $ for
$X_N((M_n)_{0 \leq n < N},(\mathcal A^-_{n, n^\prime},\mathcal A_{n, n^\prime},\mathcal A^+_{n, n^\prime})_{Ú0 \leq n,n^\prime < N})$. Alternatively, in case one wants to be more specific, one can let
$X_N((M_n)_{0 \leq n < N},A^-,A,A^+) $ denote the subshift 
$$
X_N((M_n)_{0 \leq n < N},([1,A^-_{n, n^\prime}],[1,\ A_{n, n^\prime}],[1, A^+_{n, n^\prime}])_{Ú0 \leq n,n^\prime < N}).
$$
Note that this construction carries the restriction
$$
A^-(n,n^\prime), A(n,n^\prime),A^+(n,n^\prime) \leq M_{n^\prime}, \quad 0\leq n,n^\prime < N.
$$

Given a subshift $X_N((M_n)_{0 \leq n < N},A^-,A,A^+) $,
we set
$$
\rho(\alpha^-(n, m   ) ) = 1,\quad  \rho(\alpha^+(n, m   ) ) = -1,\qquad 1 \leq m  \leq M_n,\quad 0 \leq n < N,
$$
and we let 
$$
\mathcal C_{(M_n)_{0 \leq n < N},A^-,A,A^+}(n, m ),\qquad 1 \leq m  \leq M_n,\quad0 \leq n < N,
$$
denote the circular Markov code that contains the words 
$$
c= (c_i)_{1 \leq i \leq 2I}
\in
\mathcal L (X_N((M_n)_{0 \leq n < N},A^-,A,A^+) ), \quad I \in \Bbb N,
$$
such that
$$
c_1 =\alpha^-(n, m   ) , 
$$
and
$$
\sum_{1 \leq i \leq 2I} \rho (c_i) = 0,
$$
$$
\sum_{1 \leq i \leq 2J} \rho (c_i) >0 ,\quad 1 < J < I.
$$
Setting
$$
V = \{ (n, m): 0 \leq n < N, 1 \leq m  \leq M_n \},
$$
the range and the source map are given here by setting for
$$
c= (c_i)_{1 \leq i \leq 2I}
\in
\mathcal C_{(M_n)_{0 \leq n < N},A^-,A,A^+}(n, m ), \quad I \in \Bbb N,
$$
if $  c_{2I}\in \{ \alpha^+(  n, m), 1 \leq m \leq M_n \}, $ $r(c)$ equal to $(n,1)   $, and if $  c_{1}= \alpha^-(  n, m)  $, setting 
$ s(c) $  equal to $  (n, m),1 \leq m \leq M_n,$
with the transition matrix  given by
$$
B((n, m),(  n^\prime, m^\prime) ) =
\begin{cases} 1, &\text {if $m^\prime\in \mathcal A_{n, n^\prime},$} \\
0  , &\text{if $m^\prime\notin \mathcal A_{n, n^\prime},$}  
\end{cases}  \quad0 \leq n, n^\prime < N.
$$
 Note that replacing the first symbol $\alpha^-(n, 1) $ of the words in the code 
 $$\mathcal C_{(M_n)_{0 \leq n < N},A^-,A,A^+}(n, 1 )$$ by the symbol $\alpha^-(n, m) $ yields  the words in the code 
 $$
 \mathcal C_{(M_n)_{0 \leq n < N},A^-,A,A^+}(n, m ), 0 \leq n < N, 1<m\leq M.
 $$
 
Writing
$g_n( (M_n)_{0 \leq n < N},A^-, A,A^+), 0\leq n < N,$ for
$$g( \mathcal C_{(M_n)_{0 \leq n < N}, A^-, A ,A^+}(n, m)),  \ 1 \leq m \leq M_n,
$$
we denote by $G_{(M_n)_{0 \leq n < N}, A^-, A,A^+}$ the diagonal matrix with the diagonal elements 
$$
g_n((M_n)_{0 \leq n < N}, A^-, A,A^+),\qquad0\leq n < N.
$$

 \begin{lemma} 
\begin{multline*}
g_n ((M_n)_{0 \leq n < N} A^-, A,A^+)=\\
 z^2(M_n\bold 1 + A^-{(\bold 1 -
 G_{(M_n)_{0 \leq n < N} A^-, A,A^+}A)^{-1}}
G_{(M_n)_{0 \leq n < N} A^-, A,A^+}A^+)_{n,n},\\0 \leq  n < N.
\end{multline*}

\end{lemma}
\begin{proof} 
The proof is by the transfer matrix method.
\end{proof}

In proving the next proposition we follow \cite {Ke, I, KM, Kr}.

 \begin{proposition} 
\begin{multline*}
\zeta_{X_N((M_n)_{0 \leq n < N},A^-, A,A^+)}(z) =\\
\frac{\det(1 - G_{(M_n)_{0 \leq n < N} A^-, A,A^+}A)}{\det(\bold 1 - A^-z - 
G_{(M_n)_{0 \leq n < N} A^-, A,A^+}A)\det(\bold 1 - A^+z - G_{(M_n)_{0 \leq n < N} A^-, A,A^+}A)}.
\end{multline*}
\end{proposition}
\begin{proof}
Applying Lemma (2.1), one finds that the zeta function of the neutral periodic points is given by
$$
\det(1 - G_{(M_n)_{0 \leq n < N} A^-, A,A^+}A)^{-1},
$$
the zeta function of the sets of periodic points with non-positive  multiplier by
$$
\det(\bold 1 - A^-z))^{-1}
\det(\bold 1 - A^-z - 
G_{(M_n)_{0 \leq n < N} A^-, A,A^+}A)^{-1},
$$
and the zeta function of the sets of periodic points with  non-negative multiplier by
$$
\det(\bold 1 - A^+z)^{-1}
\det(\bold 1 - A^-z - 
\det(\bold 1 - A^+z - G_{(M_n)_{0 \leq n < N} A^-, A,A^+}A)^{-1}.
$$
Taking into account that the intersection of the sets of periodic points with non-positive multplier and non-negative multiplier ia the set of neutral periodic points,  one obtains the proposition. 
\end{proof}

To consider the case of constant assignments, let $M \in \Bbb N$. Let $M, K^- , K , K^+\in \Bbb N, K^- , K , K^+\leq M,$
and denote the common value of the generating functions 
$$
g(\mathcal C_{(M)_{1\leq m \leq M},( K^-)_{0\leq n, n^{\prime}< N},(K)_{0\leq n, n^{\prime}< N},(K^+)_{0\leq n, n^{\prime}< N}  },0 \leq n < N,
$$
by $g(M, K^-, K ,K^+)$.

 \begin{proposition} 
\begin{multline*}g(M, K^-, K ,K^+)(z) = \\
\frac{1}{2KN}[1 +(MK-K^-K^+)Nz^2-\sqrt{(1 +(MK-K^-K^+)Nz^2)^2- 4MKNz^2}],
\end{multline*}
\begin{multline*}
\zeta_{X_{M ,K^-, K ,K^+}}(z) =\\
 \frac{1-KNg(M, K^-, K ,K^+)}{(1-K^-Nz -KNg(M, K^-, K ,K^+)(z))(1-K^+Nz -KNg_{M ,K^-, K ,K^+}(z))}.
\end{multline*}
\end{proposition}
\begin{proof} 
By Lemma 2.3
$$
g(z) = z^2(M + \frac{K^-NK^+g(M, K^-, K ,K^+)(z)}{1-NKg(M, K^-, K ,K^+)(z)}).
$$
Also apply Proposition 2.4.
\end{proof}

 \section{The case $N = 2$}
 We consider the case $N =2$. 
 \begin{theorem} Let there be given matrices 
 $$
 A^-= 
 ( A^-_{\delta, \delta^{\prime}}  )
 _{\delta, \delta^{\prime} \in \{0,1  \}}, \
  A= 
 ( A_{\delta, \delta^{\prime}}  )
 _{\delta, \delta^{\prime} \in \{0,1  \}}, \
  A^+= 
 ( A^+_{\delta, \delta^{\prime}}  )
 _{\delta, \delta^{\prime} \in \{0,1  \}},
 $$ 
 with entries in a commutative ring, such that
\begin{align*}
( A^-A^+)_{0,0} =( A^-A^+)_{1,1} , \tag {3.1}
\end{align*}
and
\begin{align*}
( A^-AA^+)_{0,0} =( A^-AA^+)_{1,1} . \tag {3.2}
\end{align*}
 Then
 \begin{align*}
 ( A^-A^kA^+)_{0,0} =( A^-A^kA^+)_{1,1}, \quad k \in \Bbb N. \tag {3.3}
\end{align*}
\end{theorem}
\begin{proof}
The proof is by induction. Assume that (3.1) and (3.2) hold, let $k> 1$ and assume that 
\begin{align*}
( A^-A^kA^+)_{0,0} =( A^-A^kA^+)_{1,1}.
\end{align*}
Then
$$
( A^-A^{k+1}A^+)_{0,0} =
$$
$$
 A^-(0, 0)A(0, 0)A^k(0, 0)A^+(0, 0)+ A^-(0, 0)A(0, 1)A^k(1, 0)A^+(0, 0)+
$$
$$
 A^-(0, 0)A(0, 0)A^k(0, 1)A^+(1, 0)+ A^-(0, 0)A(0, 1)A^k(1, 1)A^+(1, 0)+
$$
$$
A^-(0, 1)A^k(1, 0)A(0, 1)A^+(1, 0)+A^-(0, 1)A^k(1, 1)A(1, 1)A^+(1, 0)+
$$
$$
A^-(0, 1)A^k(1, 0)A(0, 0)A^+(0, 0)+A^-(0, 1)A^k(1, 1)A(1, 0)A^+(0, 0) =
$$
\begin{align*}
 A(0, 0)\{&A^-(0, 0)A^k(0, 0)A^+(0, 0) + \\
               &A^-(0, 0)A^k(0, 1)A^+(1, 0) + \\
               &A^-(0, 1)A^k(1, 0)A^+(0, 0)\} + 
\end{align*}
\begin{align*}
A(0, 1)A^k(1, 0)\{A^-(0, 0)A^+(0, 0)+ A^-(0, 1)A^+(1, 0) \}+      
\end{align*}
\begin{align*}
A^k(1, 1)\{&A^-(0, 0)A(0, 1)A^+(1, 0) + \\
                 &A^-(0, 1)A(1, 1)A^+(1, 0)\} \\
                 &A^-(0, 1)A(1, 0)A^+(0, 0)\} =
\end{align*}
\begin{align*}
A(0, 0)\{&A^-(1, 1)A^k(1, 1)A^+(1, 1)+ \\
             &A^-(1, 1)A^k(1, 0)A^+(0, 1)+ \\
             &A^-(1, 0)A^k(0, 1)A^+(1, 1)+ \\
             &A^-(1, 0)A^k(0, 0)A^+(0, 1)- \\
             &A^-(0, 1)A^k(1, 1)A^+(1, 0)\} +
\end{align*}             
\begin{align*}
A(0, 1)A^k(1, 0)\{A^-(1, 1)A^+(1, 1)+ A^-(1, 0)A^+(0, 1) \}+ 
\end{align*}
\begin{align*}
A^k(1, 1)\{&-A^-(0, 0)A(0, 0)A^+(0, 0)+ \\
                 &A^-(1, 1)A(1, 1)A^+(1, 1)+ \\
                 &A^-(1, 1)A(1, 0)A^+(0, 1)+ \\
                 &A^-(1, 0)A(0, 1)A^+(1, 1)+ \\
                 &A^-(1, 0)A(0, 0)A^+(0, 1)\} =
\end{align*}
$$
 A^-(1, 1)A^k(1, 1)A(1, 1)A^+(1, 1)+ A^-(1, 1)A^k(1, 0)A(0, 1)A^+(1, 1)+
$$
$$
 A^-(1, 1)A^k(1, 1)A(1, 0)A^+(0, 1)+ A^-(1, 1)A^k(1, 0)A(0, 0)A^+(0, 1)+
$$
$$
A^-(1, 0)A(0, 1)A^k(1, 0)A^+(0, 1)+A^-(1, 0)A(0, 0)A^k(0, 0)A^+(0, 1)+
$$
$$
A^-(1, 0)A(0, 1)A^k(1, 1)A^+(1, 1)+A^-(1, 0)A(0, 0)A^k(0, 1)A^+(1, 1) +
$$
\begin{align*}
A(0, 0)A^k(1, 1)\{ &A^-(1, 1)A^+(1, 1) + A^-(1, 0)A^+(1, 0)-\\
&A^-(0, 0)A^+(0, 0) - A^-(0, 1)A^+(0, 1) \}   =
\end{align*}
\begin{align*}
( A^-A^{k+1}A^+)_{1,1}.\qed
\end{align*}
\renewcommand{\qedsymbol}{}
\end{proof}
One checks that (3.3)  holds  for matrix triples   $A^-   , A  , A^+  $  that have one of the following forms
($\star$), ($\star$$\star$) or ($\star$$\star$$\star$):
\begin{align*}
A^- =
\begin{pmatrix}
A^-_{0} & A^-_{1} \\
A^-_{1}  & A^-_{0} 
\end{pmatrix}, 
\ \ A =
\begin{pmatrix}
A_{0} & A_{1} \\
A_{1}  & A_{0} 
\end{pmatrix},
\ \ A^+ =
\begin{pmatrix}
A^+_{0} & A^+_{1} \\
A^+_{0}  & A^+_{0}  
\end{pmatrix},  \tag {$\star$}
\end{align*}
\begin{align*}
A^- =
\begin{pmatrix}
A^-_{0} & A^-_{1} \\
A^-_{0}  & A^-_{1} 
\end{pmatrix}, 
\ \ A =
\begin{pmatrix}
A(0,0) & A(0,1) \\
A(1,0)  & A(1,1) 
\end{pmatrix},
\ \  \quad A^+ =
\begin{pmatrix}
A^+_{0} & A^+_{0} \\
A^+_{1}  & A^+_{1} 
\end{pmatrix},  \tag {$\star$$\star$}
\end{align*}
\begin{align*}
A^- =
\begin{pmatrix}
B_{0} & B_{1} \\
B_{2}  & B_{3} 
\end{pmatrix}, 
\ \ A =
\begin{pmatrix}
A_0 & A_1 \\
A_1  & A_2
\end{pmatrix},
\ \  \quad A^+ =
\begin{pmatrix}
B_2 & B_0 \\
B_3  & B_1 
\end{pmatrix}.
 \tag {$\star$$\star$$\star$}
\end{align*}

Denote the characteristic polynomial of a matrix $A$ by $\chi_A$.
 
 \begin{theorem}
 Let there be given an $M \in \Bbb N$, and matrices
 $$
 A^-= 
 ( A^-_{\delta, \delta^{\prime}}  )
 _{\delta, \delta^{\prime} \in \{0,1  \}}, \
  A= 
 ( A_{\delta, \delta^{\prime}}  )
 _{\delta, \delta^{\prime} \in \{0,1  \}}, \
  A^+= 
 ( A^+_{\delta, \delta^{\prime}}  )
 _{\delta, \delta^{\prime} \in \{0,1  \}},
 $$ 
 with entries in $\Bbb Z_+$, such that
  $$
  A^-(\delta, \delta^{\prime}),A(\delta, \delta^{\prime}) ,A^+(\delta, \delta^{\prime})  \leq M, \qquad \delta, \delta^{\prime} \in \{0,1  \}.
  $$
  Set
 $$
 g_0= \sum_{k\in \Bbb Z_+} g_0(k)z^k = g_0((M, M), A^-, A,A^+) , 
 $$
 $$
  g_1= \sum_{k\in \Bbb Z_+} g_1(k)z^k= g_1((M, M), A^-, A,A^+).
 $$
 Assume that
  \begin{align*}
 g_0(4) = g_1(4),  \tag {3.4}
  \end{align*}
  and
    \begin{align*}
 g_0(6) = g_1(6).  \tag {3.5}
  \end{align*}
  Then
    \begin{align*}
(A^-A^+)_{0,0} = (A^-A^+)_{1,1},  \tag {3.6}
  \end{align*}
  and
    \begin{align*}
 (A^-AA^+)_{0,0} = (A^-AA^+)_{1,1},  \tag {3.7}
  \end{align*}
and also
 \begin{align*}
 g_0 = g_1, \tag {3.8}
 \end{align*}
 and, denoting the  common value of $(A^-A^+)_{0,0}$ and $(A^-A^+)_{1,1}$ by $\eta(4)$ and the  common value of $(A^-AA^+)_{0,0}$ and $(A^-AA^+)_{1,1}$ by $\eta(6)$,  the common value $g$ of  $g_0$ and  $g_1$ satisfies the equation
 \begin{align*}
g(z) = z^2\{M + \frac{1}{g(z)^2\chi_A(g(z)^{-1})}[\eta(4) g(z) + (\eta(6) - \eta(4)\tr A)g(z)^2   ]\}.
\tag {3.6}
 \end{align*}
 \end{theorem}
 \begin{proof}
 It is 
  \begin{align*}
 g_0(2) =g_1(2) = M. 
  \end{align*}
 From (3.4)
   \begin{align*}
M(A^-A^+)_{0,0}=g(4)=M(A^-A^+)_{1,1},
   \end{align*}
 which is  (3.6). From (3.5)
   \begin{align*}
\eta(4) g(4)+ M^2 (A^-AA^+)_{0,0}=g(6)=\eta(4) g(4)+ M^2(A^-AA^+)_{1,1},
   \end{align*}
 which is (3.7).
 The proof of (3.8) is now by induction. Let $k\geq 3$, and let
 \begin{align*}
 g_0 (2q)= g_1(2q), \quad 3 < q\leq k. 
 \end{align*}
 Denote the common value of $g_0(2q)$ and $g_1(2q)$ by $g(2q), 2 \leq q \leq k$.
It is by Lemma (3.1)
   \begin{align*}
 g_{\delta_\circ} (2(k+1)) =&
 \sum_{ \delta, \delta^{\prime} \in \{0,1  \}, 1 \leq q \leq k} 
 A^-(\delta_\circ, \delta) A^q(\delta, \delta^{\prime}) A^+( \delta^{\prime},\delta_\circ) \\
& [ \sum_{\{(s(r))_{1\leq r\leq q}\in \Bbb N^{[1, q]}: \sum_{1\leq r\leq q} s(r) =k\}} \
  \prod_{1 \leq r \leq q}        g(2s(r))                      ],
\quad \delta_\circ \in\{0,1  \},
  \end{align*}
 which is (3.8).
 
One has from Lemma 2.3 that
\begin{multline*}
g_\delta =
z^2\{M  + \\
\frac{1}{1 - A(0,0)g_0  - A(1,1)g_1   +  g_0 g_1\det A}
[A^- (\delta, \delta)(1-g _{\delta^\prime} A  (\delta^\prime, \delta^\prime))g_\delta A^+  (\delta, \delta) + \\
A^- (\delta, \delta^\prime)g _{\delta^\prime} A (\delta^\prime, \delta) g_\delta A^+ (\delta, \delta )+ 
A^- (\delta, \delta)g_\delta A (\delta, \delta^\prime)g_{\delta^\prime} A^+ (\delta^\prime, \delta)+\\
A^- (\delta, \delta^\prime) (1-g_\delta A (\delta, \delta))g_{\delta^\prime} A^+ (\delta^\prime, \delta) ]      \},  \quad
(\delta,  \delta^\prime ) = (0,1), (1,0).
\end{multline*}
Apply (3.8).
\end{proof}
Setting
\begin{align*}
& B(z) = -\tr A - [M \det A +\eta(6) - \eta(4) \tr A]z^2, \\
& C(z) = 1 + [M \tr A - \eta(4)]z^2Ê\qquad
 D(z) = - Mz^2,
 \end{align*}
 one has from  equation (3.6) for the case that $\det A = 0$, that 
  \begin{align*}
g = \frac{1}{2B}(-C+  \sqrt{C^2-4BD}), \tag {3.7}
  \end{align*}
  and, also setting
  \begin{align*}
  I = B^2 - 3C \det A, \qquad
  J =9BC\det A - 27 D \det A^2 - 2B^3,
 \end{align*}
   we find by the formulas of Cardano and Vieta,  from equation (3.1) for the case that $\det A \neq 0$, and $\chi_A$ has real roots, that 
     \begin{align*}
   g = \frac{1}{3 \det A}( - B + \sqrt [3]{\tfrac{1}{2}(J + \sqrt{J^2 - 4I^2})}  +  \sqrt [3]{\tfrac{1}{2}(J - \sqrt{J^2 - 4I^2})}    ) , \tag  {3.8}
     \end{align*}
and  for the case that $\det A \neq 0$, and  $\chi_A$ has complex roots, that 
     \begin{align*}
  g =  \frac{1}{3 \det A}( - B  +2\sqrt{I} \cos (\tfrac{1}{3}(2\pi + \arccos \frac{J}{2I\sqrt{I}})). \tag {3.9}
    \end{align*}
    
\section{Excluding words from  $D_2$}

We look now more closely at sets  $\mathcal F$ of $D_2$-admissible words of lenth two, that we describe  by 0\thinspace-1 matrices 
$$
A^-_\mathcal F =
\begin{pmatrix}
A^-_\mathcal F(0,0) & A^-_\mathcal F(0,1) \\
A^-_\mathcal F(1,0)  & A^-_\mathcal F(1,1) 
\end{pmatrix}, 
$$
$$
 A_\mathcal F =
\begin{pmatrix}
A_\mathcal F(0,0) & A_\mathcal F(0,1) \\
A_\mathcal F(1,0)  & A_\mathcal F(1,1) 
\end{pmatrix},
$$
$$
 A^+ _\mathcal F=
\begin{pmatrix}
A^+_\mathcal F(0,0) & A^+_\mathcal F(0,1) \\
A^+_\mathcal F(1,0)  & A^+_\mathcal F(1,1)
\end{pmatrix},
$$
where
\begin{align*}
\mathcal F = &\{\alpha^-(\delta)\alpha^-( \delta^\prime):  \delta,  \delta^\prime \in \{0,1 \}, A^-_\mathcal F( \delta, \delta^\prime) = 0 \}
 \  \cup \\
& \{  \alpha^+(\delta)\alpha^-( \delta^\prime):   \delta,  \delta^\prime \in \{0,1 \}, A_\mathcal F( \delta, \delta^\prime) = 0  \}  
 \  \cup \\
&  \{  \alpha^+(\delta)\alpha^+( \delta^\prime):  \delta,  \delta^\prime \in \{0,1 \},  A^+_\mathcal F( \delta, \delta^\prime) = 0  \} .
\end{align*}
The subshift that is obtained by removing the words in $\mathcal F$ from $D_2$ is identical to the subshift $X(1,1, A_\mathcal F^-, A_\mathcal F,A_\mathcal F^+)$. 

We are interested in the set $\mathcal T$ of matrix triplets 
$(A^-_\mathcal F,A_\mathcal F,A^+_\mathcal F) $
such that 
$\mathcal L_2(D_2)\setminus  \mathcal F $ is the vertex set of a $D_2$-presentation, and such that
 $g_0(1,1, A_\mathcal F^-, A_\mathcal F,A_\mathcal F^+ )$ is equal to
  $g_1(1,1, A_\mathcal F^-, A_\mathcal F,A_\mathcal F^+ )$. The set  $\mathcal L_2(D_2)\setminus  \mathcal F $ is the vertex set of a $D_2$-presentation precisely if
 each row of $A^-_\mathcal F$ is non-zero, $A_\mathcal F$ is irreducible, and each column of $A^+_\mathcal F$ is non-zero.
The edge set of the $D_2$-presentation is then the set of words in $\mathcal L_3(D_2)$ with a prefix in $\mathcal L_2(D_2)\setminus  \mathcal F $, which acts as the initial vertex of the edge, and a suffix in  $\mathcal L_2(D_2)\setminus  \mathcal F $, which acts as the final vertex of the edge.
The label map $\lambda$ is given by
\begin{align*}
\lambda(\beta_{-1}\beta_0\beta_{1})=\begin{cases}
\beta_{-1}, &\text {if $ \beta_{-1}\beta_0\in \mathcal S^- \setminus \{\bold 1  \}$,}\\
\beta_{1}, &\text {if \  $ \beta_{0}\beta_1\in \mathcal S^+ \setminus \{\bold 1  \}$,}\\
\bold 1, &\text {otherwise,}
\end{cases} \qquad      \beta_{-1} \beta_0\beta_{1}\in \mathcal L(D_2).
\end{align*}

Dyck shifts have a time reversal, by which is meant a topological conjugacy between the subshift and its inverse. For $D_2$ the time reversal $T$ is given by 
\begin{align*}
T(x)_i  = \begin{cases}
\alpha_+(0), &\text {if $ x_{-i} = \alpha_-(0)$,}\\
\alpha_+(1), &\text {if $x_{-i} = \alpha_-(1)$,}\\
\alpha_-(0), &\text {if $x_{-i} = \alpha_+(0)$,}\\
\alpha_-(1), & \text{if $x_{-i} = \alpha_+(1)$,}        
\end{cases} \qquad  x = (x_i)_{i \in \Bbb Z} \in D_2.
\end{align*}
It is 
$$
T(X(1, 1,A_\mathcal F^-, A_\mathcal F^-,A_\mathcal F^+  )) = X(1, 1,   (A_\mathcal F^-)^T,  A_\mathcal F^T,(A_\mathcal F^+)^T).
$$

As an application of theorem (3.2) we  list the triplets in the set $\mathcal T$ that are neither of the form ($\star$) nor of the form 
($\star$$\star$) nor ($\star$$\star$$\star$),  choosing a representative out of every set of triplets that can be obtained from one another by exchanging the indices 0 and 1 and/or by time reversal:
 For
 \begin{align*}
 \ \ A_\mathcal F =
\begin{pmatrix}
1 & 1 \\
1 & 0 
\end{pmatrix},
 \end{align*}
we   take
\begin{align*}
A_\mathcal F^- =
\begin{pmatrix}
1 & 1 \\
1 & 0
\end{pmatrix}, 
\ \ A_\mathcal F^+ =
\begin{pmatrix}
0& 1\\
1 & 0
\end{pmatrix},  
\end{align*}
and for
\begin{align*}
 A_\mathcal F =
\begin{pmatrix}
0 & 1 \\
1  & 0 
\end{pmatrix},
\ \ A_\mathcal F =
\begin{pmatrix}
1 & 1 \\
1  & 1 
\end{pmatrix},
\end{align*}
we take
\begin{align*}
A_\mathcal F^- =
\begin{pmatrix}
1 & 1 \\
1 & 1
\end{pmatrix}, 
\ \ A_\mathcal F^+ =
\begin{pmatrix}
1 & 0\\
0 & 1 
\end{pmatrix},
\end{align*}
\begin{align*}
A_\mathcal F^- =
\begin{pmatrix}
1 & 1 \\
0 & 1
\end{pmatrix}, 
\ \ A_\mathcal F^+ =
\begin{pmatrix}
1& 1\\
0 & 1 
\end{pmatrix}, 
\end{align*}
\begin{align*}
A_\mathcal F^- =
\begin{pmatrix}
1 & 1 \\
1 & 1
\end{pmatrix}, 
\ \ A_\mathcal F^+ =
\begin{pmatrix}
0 & 1 \\
1 & 0
\end{pmatrix}, 
\end{align*}
\begin{align*}
A_\mathcal F^- =
\begin{pmatrix}
1 & 1 \\
1 & 0
\end{pmatrix}, 
\ \ A_\mathcal F^+ =
\begin{pmatrix}
0 & 1 \\
1 & 1
\end{pmatrix}. 
\end{align*}

The zeta functions of the periodic points with negative and with positive multipliers, and of the neutral periodic points of the  $\mathcal D_2$-presentations that arise from the triplets in the set $\mathcal A$ can be obtained from Proposition 2.5 using the formulas (3.7),  (3.8) and (3.9). In a number of cases they can also be determined by direct inspection without the use of circular codes that are properly Markov, that is, without recourse to Lemma 2.1.

\medskip

 \end{document}